\documentclass[11pt]{amsart}

\usepackage{amsmath, amssymb, amsthm, amscd}
\usepackage{url}
\usepackage{graphicx}
\usepackage[hidelinks,pagebackref,pdftex]{hyperref}
\usepackage{color}
\usepackage{import}
\usepackage{microtype}
\usepackage{geometry}
\usepackage{marginnote}
\long\def\@savemarbox#1#2{\global\setbox#1\vtop{\hsize\marginparwidth 
  \@parboxrestore\tiny\raggedright #2}}
\renewcommand*{\backref}[1]{}
\renewcommand*{\backrefalt}[4]{
  \ifcase #1
  [No citations.]
  \or [#2]
  \else [#2]
  \fi }

\AtBeginDocument{%
   \def\MR#1{}
}

\numberwithin{equation}{section}
\theoremstyle{plain}
\newtheorem{theorem}[equation]{Theorem}

\newtheorem{lemma}[equation]{Lemma}

\newtheorem{corollary}[equation]{Corollary}
\newtheorem{proposition}[equation]{Proposition}

\newtheorem*{namedtheorem}{\theoremname}
\newcommand{\theoremname}{testing}
\newenvironment{named}[1]{\renewcommand{\theoremname}{#1}\begin{namedtheorem}}{\end{namedtheorem}}

\theoremstyle{definition}
\newtheorem{definition}[equation]{Definition}
\newtheorem{remark}[equation]{Remark}
\newtheorem{question}[equation]{Question}

\newcommand{\from}{\colon} 
\newcommand{\HH}{{\mathbb{H}}}
\newcommand{\RR}{{\mathbb{R}}}
\newcommand{\ZZ}{{\mathbb{Z}}}

\newcommand{\CC}{{\mathbb{C}}}

\newcommand{\R}{{\mathbb{R}}}
\newcommand{\Z}{{\mathbb{Z}}}

\newcommand{\C}{{\mathbb{C}}}

\newcommand{\calC}{{\mathcal{C}}}
\newcommand{\calH}{{\mathcal{H}}}
\newcommand{\calR}{{\mathcal{R}}}

\newcommand{\bdy}{\partial}

\renewcommand{\setminus}{{\smallsetminus}}

\newcommand{\PSL}{\operatorname{PSL}}
\newcommand{\UHS}{\HH^3}
\newcommand{\AH}{\operatorname{AH}}
\newcommand{\MP}{\operatorname{MP}}

\newcommand{\refthm}[1]{Theorem~\ref{Thm:#1}}
\newcommand{\reflem}[1]{Lemma~\ref{Lem:#1}}
\newcommand{\refprop}[1]{Proposition~\ref{Prop:#1}}
\newcommand{\refcor}[1]{Corollary~\ref{Cor:#1}}

\newcommand{\refdef}[1]{Definition~\ref{Def:#1}}

\newcommand{\reffig}[1]{Figure~\ref{Fig:#1}}

\title{Cusp shape and tunnel number}
\author{Vinh Dang}
\address{Lone Star College, North Harris, Houston, TX 77073, USA} 
\email{vinh.x.dang@lonestar.edu}
\author{Jessica S. Purcell}
\address{School of Mathematical Sciences, Monash University, VIC 3800, Australia}
\email{jessica.purcell@monash.edu}

\begin{document}

\begin{abstract}
We show that the set of cusp shapes of hyperbolic tunnel number one manifolds is dense in the Teichm\"uller space of the torus. A similar result holds for tunnel number~$n$ manifolds. As a consequence, for fixed $n$, there are infinitely many hyperbolic tunnel number~$n$ manifolds with at most one exceptional Dehn filling. This is in contrast to large volume Berge knots, which are tunnel number one manifolds, but with cusp shapes converging to a single point in Teichm\"uller space. 
\end{abstract}

\maketitle

\section{Introduction}\label{Sec:Intro}
This paper is concerned with two invariants for a 3-manifold with torus boundary.

The first is geometric. If the interior of the 3-manifold admits a hyperbolic structure, then a tubular neighborhood of the torus boundary inherits a Euclidean structure, and this Euclidean structure is unique up to similarity of the torus. The similarity class of such a Euclidean structure is called the \emph{cusp shape} of the 3-manifold. Manifolds with restricted cusp shapes have restrictions on their Dehn fillings, for example by the $2\pi$-theorem~\cite{BleilerHodgson} or the 6-Theorem~\cite{Agol:Bounds, Lackenby:Word}. There are uncountably many similarity classes of Euclidean metrics on the torus, but only countably many finite volume hyperbolic 3-manifolds. Thus not all similarity classes of metrics can appear as cusp shapes of hyperbolic 3-manifolds. Nevertheless, in~\cite{Nim-94}, Nimershiem proved that the set of similarity classes occurring as cusp shapes of any hyperbolic 3-manifold is dense. Does the set of similarity classes occurring as cusp shapes remain dense when other restrictions are put on the manifold? 

The second notion is topological. Every 3-manifold $M$ admits a Heegaard splitting: a decomposition along an embedded Heegaard surface into two handlebodies, or compression bodies if $M$ has boundary. Frequently, we would like to know the minimal genus over all possible Heegaard splittings of $M$. If $M$ has a single torus boundary component, a genus $g$ Heegaard splitting gives a decomposition of $M$ into a genus $g$ handlebody and a compression body homeomorphic to the union of $T^2\times I$ and $g-1$ one-handles; see \refdef{ComprBody} below. In this case, the core of each one-handle of the compression body is called a \emph{tunnel}; the union of the $g-1$ tunnels in $M$ is a \emph{tunnel system for $M$}. A 3-manifold with one torus boundary component and a minimal genus $g$ Heegaard splitting is called a \emph{tunnel-number $g-1$} manifold. For hyperbolic manifolds, the tunnel number is known to give a lower bound on hyperbolic volume~\cite{Thurston-79}.

A natural question to ask is whether there is a relationship between cusp shape and tunnel number, or whether restrictions on tunnel number give restrictions on cusp shape. In the case of tunnel number one, work of Baker on Berge knots~\cite{Baker:LargeVolBerge} implies that there is an infinite sequence of cusped hyperbolic tunnel number one manifolds with cusp shape approaching that of a $1\times 2$ rectangle, up to scale. Do all hyperbolic tunnel number one manifolds have a restricted set of cusp shapes in their limit set? In this paper, we show they do not. In fact, the set of cusp shapes is dense in the Teichm\"uller space of the torus. 

\begin{theorem}\label{Thm:Tunnel1CuspShapes}
For any $\epsilon>0$ and any point $[T]$ in the Teichm\"uller space of the torus, there exists a complete, finite volume, tunnel number one hyperbolic manifold $M$ with a single cusp whose cusp shape is $\epsilon$-close to $[T]$. 
\end{theorem}

A similar result holds for tunnel number $n$ manifolds.

\begin{theorem}\label{Thm:TunnelNCuspShapes}
  For any $\epsilon>0$, any point $[T]$ in the Teichm\"uller space of the torus, and any integer $n\geq 1$, there exists a complete, finite volume hyperbolic manifold $M$ that admits a tunnel system of $n$ tunnels and has a single cusp, with cusp shape $\epsilon$-close to $[T]$. 
 Further, assuming recent work of Maher and Schleimer~\cite{MaherSchleimer}, we can ensure $M$ is tunnel number $n$.
\end{theorem}


\refthm{Tunnel1CuspShapes} also has interesting Dehn filling consequences.

\begin{named}{\refcor{TunnelNExceptions}}
For any integer $n\geq 1$, there exist infinitely many hyperbolic tunnel number~$n$ manifolds with at most one exceptional Dehn filling.
\end{named}

The Berge knots form a class of tunnel number one knots in $S^3$ that admit a lens space Dehn filling. Thus these knots have at least two exceptional Dehn fillings: the lens space and $S^3$. Baker showed that infinitely many come from Dehn filling a minimally twisted chain link. It follows from his work that the cusp shapes of these knots limit to a single Euclidean similarity structure; see \refprop{Berge}. In sharp contrast, \refthm{Tunnel1CuspShapes} and \refcor{TunnelNExceptions} imply that cusp shapes of more general tunnel number one manifolds have broader limits, and fewer exceptional Dehn fillings.
Indeed, \refcor{TunnelNExceptions} was already known in the case $n=1$: The 2-bridge knots form a class of tunnel number one manifolds. It is known that infinitely many hyperbolic 2-bridge knots admit at most one exceptional Dehn fillings, due to work of Futer, Kalfagianni and Purcell bounding the cusp areas of these knots~\cite{FKP:Farey}. However, that paper does not discuss density of cusp shapes, and does not give any information on knots with higher tunnel numbers.

There is less literature on cusp shapes or geometric limits of tunnel number $n$ manifolds for higher $n$. Examples of such manifolds with at least two exceptional fillings are known. Infinite families of one-cusped hyperbolic manifolds with tunnel number $2$ and at least two exceptional Dehn fillings have been constructed by Dunfield, Hoffman, and Licata~\cite{DunfieldHoffmanLicata}, and by Baker~\cite{Baker:Tunnel2}. 
Earlier, Teragaito found a family of knots in $S^3$ with a small Seifert fibred surgery~\cite{Teragaito}; these were generalised by Baker, Gordon, and Luecke to produce knots with tunnel number at least two and at least two exceptional Dehn fillings~\cite{BakerGordonLueckeIII}. More generally, for any $n$, Eudave-Mu{\~n}oz and Luecke have constructed hyperbolic knots in $S^3$ with two exceptional Dehn fillings and tunnel number bounded below by $n$~\cite{EudaveMunozLuecke}. 
By the 6-Theorem~\cite{Agol:Bounds, Lackenby:Word}, the cusp shapes of these manifolds must lie in a bounded region of the moduli space of the torus. However, the limiting shapes of cusps of these families of manifolds have not been analyzed.

To investigate results on Dehn filling further, for example to determine whether there are tunnel number $n$ manifolds with \emph{no} exceptional Dehn fillings, we need to know the cusp \emph{area} of our manifolds, i.e.\ the area of the boundary of a maximal horoball neighborhood of the cusp. In this paper we use an algebraic limit argument to give cusp shapes, bounding the similarity structure of the cusp torus, but not its scale. Unfortunately, our method of proof does not give any information on cusp area; a geometric limit argument would be required. Thus while the shortest slope of the torus must have length at least one, and therefore long narrow tori have at most one exceptional Dehn filling, we cannot determine whether fatter tori are large or small, with long or short slope lengths. The following are still unknown.

\begin{question}
  Are there families of tunnel number $n$ manifolds with no short slopes and no exceptional Dehn fillings? What are the cluster points of the areas of the maximal cusps of tunnel number $n$ manifolds?
\end{question}

To prove \refthm{Tunnel1CuspShapes}, we consider geometric structures on compression bodies, similar to work in~\cite{CooperLackenbyPurcell}, \cite{BurtonPurcell}, and~\cite{LackenbyPurcell:ComprBodies}. Given any framed Euclidean similarity structure on a torus, we first find hyperbolic structures on compression bodies with that cusp shape $[T]$, then show these lead to maximally pinched structures with cusp shape approaching $[T]$. By gluing maximally pinched compression bodies to maximally pinched handlebodies and Dehn filling, we obtain our main results.

\subsection{Acknowledgements}
The authors thank Ken Baker for helpful conversations, and the referee for  suggestions to improve the paper. 

\section{Hyperbolic structures on compression bodies}\label{Sec:ComprBodies}

Let $S$ be a closed, orientable (possibly disconnected) surface with genus at least $1$, and let $I$ denote the interval $[0,1]$.

\begin{definition}\label{Def:ComprBody}
  A \emph{compression body} $C$ is either a handlebody, or the result of attaching $1$-handles to $S\times I$ along $S\times\{1\}$ such that the result is connected. The \emph{negative boundary} is $S\times\{0\}$ and is denoted $\partial_{-}C$; if $C$ is a handlebody then $\bdy_-C=\emptyset$. The \emph{positive boundary} is $\partial C\setminus\partial_{-}C$ and is denoted $\partial_{+}C$.
\end{definition}

The compression bodies that we study here are those for which $\bdy_-C$ is a torus and $\bdy_+C$ is a genus $(n+1)$ surface, with $n\geq 1$ an integer. We call this the $(1;n+1)$-compression body and denote it by $C(1;n+1)$ or just $C$ when there is no confusion. The numbers $1$ and $n+1$ refer to the genera of the boundary components. Thus $C(1;n+1)$ is obtained by attaching $n$ 1-handles to $\mathbb{T}^{2}\times I$. See \reffig{CompressionBody}.

\begin{figure}
\centering
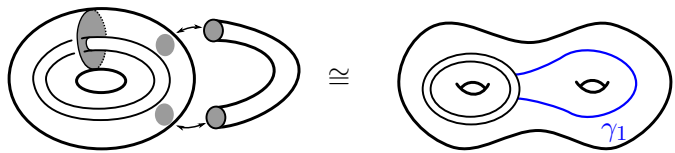
\caption{The $(1;2)$-compression body}
\label{Fig:CompressionBody}
\end{figure}

The fundamental group of $C(1;2)$ is $\displaystyle \pi_{1}(C(1;2)) = (\Z\times\Z)*\Z$, because $C(1;2)$ deformation retracts to the union of a torus and an arc $\gamma_1$ at the core of the 1-handle; see \reffig{CompressionBody}. This arc is called the \emph{core tunnel} of $C(1;2)$. We will denote the generators of the $\Z\times\Z$ factor by $\alpha$ and $\beta$, and the generator of the free $\Z$ by $\gamma_1$.

More generally, for the $(1;n+1)$-compression body, the fundamental group is
\[
\pi_1(C(1;n+1)) = (\ZZ\times\ZZ)*\ZZ*\dots* \ZZ,
\]
with generators $\langle\alpha,\beta\rangle$ for the $\ZZ\times\ZZ$ factor, and generators $\gamma_1, \dots, \gamma_n$ for the free $\ZZ$ factors, where $\gamma_{1}, \ldots, \gamma_{n}$ can be taken to be freely homotopic to the core tunnels of the 1-handles, with endpoints connected by an arc on $T^2\times \{1\}$.

\subsection{Hyperbolic structures}
We will construct complete hyperbolic structures on compression bodies with desired geometric properties.

Let $\rho\from \pi_1(C)\to\PSL(2,\CC)$ be a discrete, faithful representation. Let $\Gamma = \rho(\pi_1(C))$ denote the image of such a representation. Note $\Gamma$ acts on $\HH^3$. The representation $\rho$ gives a \emph{hyperbolic structure} on $C$ if $\HH^3/\Gamma$ is homeomorphic to the interior of $C$.

Note $\Gamma =\rho(\pi_1(C))$ has a subgroup $\Gamma_\infty := \langle\rho(\alpha), \rho(\beta)\rangle \cong \ZZ\times\ZZ$, generated by parabolics fixing a point at infinity. By conjugating $\Gamma$ if necessary, we may assume $\Gamma_\infty$ fixes the point $\infty$ in $\HH^3$. Let $H$ be any horoball about infinity. Then $H/\Gamma_\infty$ is a \emph{rank-2 cusp} of $C$. Its boundary $\bdy H/\Gamma_\infty$ is a flat (Euclidean) torus.

We will study $\Gamma$ by studying its isometric spheres, as follows.

\begin{definition}\label{Def:IsoSphere}
  Let $g\in \PSL(2,\CC)$. Then $g$ can be represented by a $2\times 2$ matrix with determinant $1$ and coefficients in $\CC$, well defined up to multiplication by $\pm 1$:
  \[ g = \pm \begin{pmatrix}a&b\\c&d \end{pmatrix}.\]
  Provided $g$ does not fix the point $\infty$, the \emph{isometric sphere} of $g$, denoted by $I(g)$, is the Euclidean hemisphere centered at $g^{-1}(\infty) = -d/c$ with radius $1/|c|$.
\end{definition}

The following properties of isometric spheres are well known. For a reference, see for example~\cite[Section~2]{LackenbyPurcell:ComprBodies}.

\begin{lemma}\label{Lem:IsoSphereProps}
Let $g\in\PSL(2,\CC)$. 
\begin{enumerate}
\item If $H$ is any horosphere centered at $\infty$ in the upper half space $\UHS$, then $I(g)$ is the set of points in $\UHS$ equidistant from $H$ and $g^{-1}(H)$.
\item $I(g)$ is mapped isometrically to $I(g^{-1})$ by $g$; the half ball $B(g)$ bounded by $I(g)$ is mapped to the exterior of the half ball $B(g^{-1})$ bounded by $I(g^{-1})$.
\end{enumerate}
\end{lemma}

\begin{definition}\label{Def:VertFundDomain}
  A \emph{vertical fundamental domain} for $\Gamma_{\infty}$ is a fundamental domain for the action of $\Gamma_{\infty}$ cut out by finitely many vertical geodesics planes in $\UHS$.
\end{definition}

Examples are in order. 

\begin{lemma}\label{Lem:CuspedStructure}
Let $\rho\from \pi_{1}(C)\to \PSL(2,\C)$ be defined on generators $\alpha$, $\beta$ and $\gamma_1$ of $\pi_{1}(C)$ by
\[\rho(\alpha) = \begin{pmatrix}1&a\\0&1\end{pmatrix}, \quad
  \rho(\beta) = \begin{pmatrix}1&b\\0&1\end{pmatrix}, \quad
    \rho(\gamma_1) = \begin{pmatrix}2&-1\\1&0\end{pmatrix},\]
where $a$ and $b$ are complex numbers which are linearly independent over $\R$ with $|a| > 4$ and $|b| > 4$. Then $\rho$ gives a hyperbolic structure on $C(1;2)$.
\end{lemma}

\begin{proof}
The isometric sphere $I(\rho(\gamma_1))$ has center $0$ and radius $1$. The isometric sphere $I(\rho(\gamma_1^{-1}))$ has center $2$ and radius $1$. A vertical fundamental domain for the action of $\Gamma_\infty = \langle\rho(\alpha),\rho(\beta)\rangle$ is cut out by vertical planes through lattice points generated by translation by $a$ and $b$ in $\CC$. Because $|a|>4$ and $|b|>4$, we may choose the planes so that a vertical fundamental domain contains both isometric spheres $I(\rho(\gamma_1))$, $I(\rho(\gamma_1^{-1}))$ in its interior.

We claim that the four planes of a vertical fundamental domain and the two isometric spheres $I(\rho(\gamma_1))$ and $I(\rho(\gamma_1^{-1}))$ form a fundamental region for $C(1;2)$. This can be seen by the Poincar{\'e} polyhedron theorem (e.g.~\cite[Theorem~2.25]{LackenbyPurcell:ComprBodies}): vertical planes are glued isometrically to vertical planes, $I(\rho(\gamma_1))$ is glued isometrically to $I(\rho(\gamma_1^{-1}))$. Thus the object obtained by these face pairings is a manifold with discrete fundamental group isomorphic to the face pairing group, isomorphic to $\langle \rho(\alpha), \rho(\beta), \rho(\gamma_1)\rangle \cong \pi_1(C)$.

Finally, to see that it gives a hyperbolic structure on $C(1;2)$, we need to show that the interior of $C(1;2)$ is homeomorphic to $\HH^3/\rho(\pi_1(C))$. This can be seen as follows. The quotient of $\HH^3$ by $\rho(\langle\alpha,\beta\rangle)$ is homeomorphic to $T^2\times\RR$. Remove the open half-spaces cut out by $I(\rho(\gamma_1))$ and $I(\rho(\gamma_1^{-1}))$ and glue these to each other. The result is homeomorphic to attaching a 1-handle. Thus the space is homeomorphic to the interior of $C$.
\end{proof}

\reffig{CuspedStructure} shows some isometric spheres of the hyperbolic structure of \reflem{CuspedStructure}.

\begin{figure}
\centering
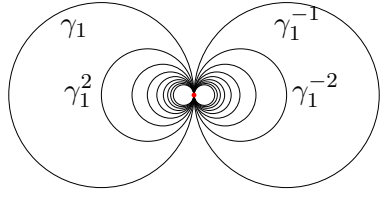
\caption{Isometric spheres for some powers of $\rho(\gamma_1)$ and $\rho(\gamma_1^{-1})$, for $\Gamma$ as in \reflem{CuspedStructure}. Note we  have labeled the sphere $I(\rho(\gamma_1))$ simply by $\gamma_1$ above, and similarly for other group elements.}
\label{Fig:CuspedStructure}
\end{figure}

The following extends the example to $C(1;n+1)$.

\begin{lemma}\label{Lem:CuspedStructN}
Fix an integer $n\geq 1$. Define $\rho(\gamma_1)$ as in \reflem{CuspedStructure}. Let
\[ A = \begin{pmatrix}1& 5 \\ 0 & 1\end{pmatrix}. \]
Let $\rho(\gamma_2) = A\rho(\gamma_1)A^{-1}$, and inductively define $\rho(\gamma_{i+1}) = A\rho(\gamma_i)A^{-1}$.
Finally, let $a$ and $b$ be complex numbers that are linearly independent over $\RR$, with $|a|>5$ and $|b|>5$ chosen sufficiently large so that the set of isometric spheres $\{I(\rho(\gamma_j^{\pm 1}))\}$ is disjoint from the translates of that set under $\rho(\alpha)$ and $\rho(\beta)$. 
Then $\rho$ gives a complete hyperbolic structure on $C(1;n+1)$. 
\end{lemma}

\begin{proof}
Let $H$ be a horosphere centered at $\infty$ in $\UHS$. Note that using \reflem{IsoSphereProps}, 
\begin{align*}
  I(A\rho(\gamma_{1})A^{-1}) &= \{x\in\UHS: d(x,H) = d(x,A\rho(\gamma_{1}^{-1})A^{-1}H)\} \\
  &= \{x\in\UHS: d(x,H) = d(x,A\rho(\gamma_{1}^{-1})H)\}\\
  &= \{x\in\UHS: d(A^{-1}x,H) = d(A^{-1}x,\rho(\gamma_{1}^{-1})H)\}\\
  &= A\{y\in\UHS: d(y,H) = d(y,\rho(\gamma_{1}^{-1})H)\}\\
  &= AI(\rho(\gamma_{1})).
\end{align*}  
The above calculation shows that the isometric sphere of the element $A\rho(\gamma_{1})A^{-1}$ is obtained from that of $\rho(\gamma_{1})$ by horizontal translation by $5$. Similarly, for any integer $k$, the isometric spheres of $A\rho(\gamma_{1}^{k})A^{-1}$  are obtained from the corresponding spheres of $\rho(\gamma_{1}^{k})$ by horizontal translation by $5$. Thus, the configuration of spheres  $\rho(\gamma_{1}^{k})$ in \reffig{CuspedStructure} does not intersect with the $A$-translated configuration. Thus isometric spheres for $\rho(\gamma_i)$ do not overlap for any $i$.

As in the proof of the previous lemma, the Poincar{\'e} polyhedron theorem implies that these isometric spheres, as well as a vertical fundamental domain, cut out a fundamental region for the action of $\Gamma = \rho(\pi_1(C(1;n+1)))$. Thus this is a discrete, faithful representation. Moreover, again gluing of isometric spheres $I(\rho(\gamma_i))$ and $I(\rho(\gamma_i^{-1}))$ corresponds to attaching one-handles. Thus the manifold $\HH^3/\Gamma$ is homeomorphic to the interior of $C(1;n+1)$. 
\end{proof}

\section{Spaces of hyperbolic structures}\label{Sec:Spaces}

Define the \emph{representation space} of $C$ with \emph{parabolic locus} $T$ to be the space
\[\calR(C,T) = \{\mbox{representations }\rho\from\pi_{1}(C)\to\PSL(2,\C) \mid g\in\pi_1(T)\Rightarrow \rho(g)\mbox{ is parabolic} \}.\]
The space $\AH(C,T)$ of hyperbolic structures on $C$ is the set 
\[\AH(C,T) = \{\rho\in\calR(C,T) \mid \rho\mbox{ is discrete and faithful}\}/\sim\]
where the equivalence $\sim$ is given by $\rho\sim\rho'$ if they are conjugate representations. 

The topology on $\AH(C,T)$ is given by \emph{algebraic convergence}, as follows.

\begin{definition}\label{Def:AlgConvergence}
A sequence of representations $\{\rho_{n}\}$ in $\AH(C,T)$ converges \emph{algebraically} to a representation $\rho$ if the sequence $\{\rho_{n}(g)\}$ converges to $\rho(g)$ for every $g\in\pi_{1}(C)$. Here we regard $\rho_{n}(g)$ as a point in the manifold $\PSL(2,\C)$, and so $\{\rho_{n}(g)\}$ converges to $\rho(g)$ in the sense of convergence of points in $\PSL(2,\C)$.
\end{definition}

Suppose $\{\rho_{n}\}$ converges to $\rho$ algebraically. Algebraic convergence does not reveal much geometric information.
For example the sequence of quotient manifolds $\UHS/\rho_n(\pi_1(C))$ does not necessarily converge to the limiting quotient manifold $\UHS/\rho(\pi_1(C))$; see~\cite{Thurston-79, Brock-01, KerckhoffThurston}. A stronger notion is \emph{geometric convergence}.

\begin{definition}\label{Def:GeomConvergence}
A sequence of subgroups $\Gamma_{n}$ of $\PSL(2,\C)$ converges \emph{geometrically} to a subgroup $\Gamma$ if:
\begin{enumerate}
\item For each $\zeta\in\Gamma$, there exists $\zeta_{n}\in\Gamma_{n}$ such that $\{\zeta_{n}\}$ converges to $\zeta$.
\item If $\zeta_{n_{j}}\in \Gamma_{n_{j}}$, and $\{\zeta_{n_{j}}\}$ converges to $\zeta$, then $\zeta\in \Gamma$.
\end{enumerate}
\end{definition}

Geometric convergence of discrete groups implies convergence of their quotient spaces; see e.g.~\cite[Theorem~3.2.9]{CanaryEpsteinGreen:Notes}. 
Work of Jorgensen and Marden~\cite{Marden-Jorgensen} implies that if a sequence $\{\Gamma_n\}$ converges to a discrete group $\Gamma$ algebraically, then there is a subsequence converging geometrically to a group $\Gamma'$, with $\Gamma\leq\Gamma'$. Thus the quotient manifold $M=\UHS/\Gamma$ of the algebraic limit is a cover of the quotient manifold $M'=\UHS/\Gamma'$ of the geometric limit. 

\begin{definition}\label{Def:GeomFinite}
A hyperbolic structure is \emph{geometrically finite} if $\UHS/\Gamma$ admits a finite-sided, convex fundamental domain.
\end{definition}

For example, the hyperbolic structures of Lemmas~\ref{Lem:CuspedStructure} and~\ref{Lem:CuspedStructN} are geometrically finite. The fundamental domain consists of the isometric spheres described in the lemmas as well as a vertical fundamental domain. These are finite sided. 

\begin{definition}\label{Def:MinParabolic}
A hyperbolic structure on $C$ is \emph{minimally parabolic} if $\Gamma$ has no non-peripheral parabolic subgroups. Equivalently, in the case of $C(1;n+1)$, a hyerbolic structure is minimally parabolic if it has no rank-1 parabolic subgroup aside from subgroups of $\langle\rho(\alpha),\rho(\beta)\rangle$.  \end{definition}

The hyperbolic structures of Lemmas~\ref{Lem:CuspedStructure} and~\ref{Lem:CuspedStructN} are \emph{not} minimally parabolic. Each of the representatives $\rho(\gamma_i)$ has been chosen to be parabolic, but $\gamma_i$ is not an element of the free abelian rank-2 subgroup $\pi_1(T)$ of $\pi_1(C)$. Thus $\rho$ gives $C$ a structure which is geometrically finite but not minimally parabolic.

Denote the space of minimally parabolic, geometrically finite structures on $C$ by $\MP(C,T)$. 
The boundary of the space $\MP(C,T)$ consists of equivalence classes of representations which give $C$ structures that are geometrically finite but not minimally parabolic, or structures that are \emph{geometrically infinite} (not geometrically finite). The Density theorem (\cite{NamaziSouto}, \cite{Ohshika}) implies that every structure on the boundary of $\MP(C,T)$ is the algebraic limit of a convergent sequence in $\MP(C,T)$, and that the closure of $\MP(C,T)$ is $\AH(C,T)$.
We call structures that are geometrically finite but not minimally parabolic \emph{pinched structures}.

Note that \reflem{CuspedStructure} and \reflem{CuspedStructN} give examples of  pinched structures.

\begin{definition}\label{Def:MaxPinched}
A \emph{maximally pinched structure} is a pinched structure $\rho$ for which
the boundary of the convex core of $\HH^3/\rho(\pi_1(C))$ consists of 3-punctured spheres. 
\end{definition}

\begin{remark}
The structures of \refdef{MaxPinched} are typically called \emph{maximal cusps} in the literature on Kleinian groups, because the boundary of each 3-punctured sphere is a rank-1 cusp (e.g.~\cite{McMullen:MaxCusps, CCHS:MaxCusps}). Unfortunately, in the case of manifolds with torus boundary components, a \emph{maximal cusp} refers to something else entirely: a maximally embedded horoball neighborhood of all rank-2 cusps of the manifold. Because our maximally pinched manifolds will also have a maximal horoball neighborhood of the torus boundary component, we rename the type of structures. 
We decided on \emph{pinched} because such a structure pinches a maximal collection of curves on the higher genus boundary component into parabolics, with length $0$.
\end{remark}

Maximally pinched structures are useful because of the following result due to Canary, Culler, Hersonsky, and Shalen~\cite{CCHS:MaxCusps}, extending work of McMullen~\cite{McMullen:MaxCusps}. 

\begin{theorem}\label{Thm:CCHS}
Maximally pinched structures are dense on the boundary $\partial \MP(C,T)$ of the space $\MP(C,T)$.\qed
\end{theorem}

\begin{lemma}\label{Lem:CuspedStructwithShape}
Let $[T]$ be any point in the Teichm\"uller space of the torus. For any $n\geq 1$, there exists a pinched structure $\rho$ on $C(1;n+1)$ such that $\rho$ gives $C(1;n+1)$ cusp shape $[T]$.
\end{lemma}

\begin{proof}
The representation $\rho\from \pi_{1}(C)\to \PSL(2,\C)$ is defined as in \reflem{CuspedStructure} and in \reflem{CuspedStructN}, choosing complex numbers $a$ and $b$ to give the desired shape, as follows:
\begin{enumerate}
\item[(i)] Complex numbers $a$ and $b$ are chosen to be linearly independent over $\R$ with $|a| > 5$ and $|b| > 5$, and $|a|$ and $|b|$ large enough that the set of isometric spheres $\{I(\rho(\gamma_j^{\pm 1}))\}$ is disjoint from its translates under $\rho(\alpha)$ and $\rho(\beta)$. 
\item[(ii)] Also, $[\bdy H/\langle \rho(\alpha),\rho(\beta)\rangle] = [T]$ for $H$ a horoball about infinity. 
\end{enumerate}
Condition (i) ensures that $a$ and $b$ satisfy the hypothesis of \reflem{CuspedStructN}. Hence, $\rho$ gives a pinched structure. Condition (ii) ensures that the shape on the rank-2 cusp of $C$ equipped with this structure is $[T]$.
\end{proof}

Recall that the Teichm\"uller space of the torus is the space of similarity classes of (marked) flat metrics on the torus. It is isometric to the upper half plane $\HH^2 = \{z\in\CC\mid {\rm Im}(z)>0\}$, equipped with the hyperbolic metric; see for example~\cite{Lehto}. 

\begin{lemma}\label{Lem:maxCuspOnC}
Let $[T]$ be any point in the Teichm\"uller space of the torus. For any $\epsilon > 0$, there exists a maximally pinched structure $\rho'$ on $C(1;n+1)$ such that the cusp shape of $\rho'$ is $\epsilon/2$-close to $[T]$.
\end{lemma}

\begin{proof}
Let $\rho$ be a pinched structure on $C$ with cusp shape $[T]$ as in \reflem{CuspedStructwithShape}. By \refthm{CCHS}, maximally pinched structures are dense on the boundary $\partial \MP(C,T)$ of the space $\MP(C,T)$. Therefore, the fact that $\rho\in \partial MP(C)$ implies that there exists a sequence $\{\rho_{n}\}$ of maximally pinched structures on $C$ which converges algebraically to $\rho$. Thus, there exists a maximally pinched structure $\rho'$ that is $\epsilon/2$-close to $\rho$. In particular, $\rho'(\alpha)$ and $\rho'(\beta)$ are $\epsilon/2$-close to $\rho(\alpha)$ and $\rho(\beta)$ because the convergence is algebraic. As a result, $\rho'$ gives $C$ a cusp shape that is $\epsilon/2$-close to $[T]$.
\end{proof}

Lemma~\ref{Lem:maxCuspOnC} is only an existence result, and unfortunately gives no explicit structure or examples to work with. Finding examples of maximally pinched structures is non-trivial. To find such a structure on $C(1;2)$, a choice of three appropriate curves on $\bdy_+C(1;2)$ must be pinched, which amounts to finding values of $\rho(\alpha)$, $\rho(\beta)$ and $\rho(\gamma_1)$ for which certain products of generators have trace $\pm 2$, namely those products that represent the boundary curves. For example, one maximally pinched structure with cusp shape $[T]=[-1/4+i\sqrt{3}/4]$ is given by the representation
\[
\rho(\alpha) = \begin{pmatrix}1&4\\0&1\end{pmatrix}, \quad
\rho(\beta) = \begin{pmatrix} 1&-1+i\sqrt{3}\\0&1 \end{pmatrix}, \quad
\rho(\gamma_1) = \begin{pmatrix} 2&-1\\1&0 \end{pmatrix}.
\]
This maximally pinched structure has a fundamental region in $\HH^3$ lying over the parallelogram and hemispheres on $\CC$ shown on the left of \reffig{MaxPinched}; the group elements labeling the circles in that figure are short for isometric spheres $I(\rho(\gamma_1))$, $I(\rho(\gamma_1^{-1}))$, etc. To obtain a fundamental region for the convex core, cut out four additional hemispheres that are ``dual,'' enclosing each curvilinear triangle in the original fundamental region as shown on the right of \reffig{MaxPinched}.
\begin{figure}
  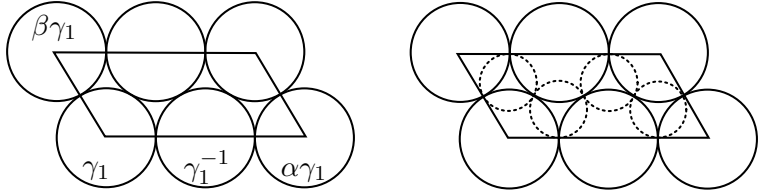
  \caption{Left: A fundamental region for a maximally pinched structure on $C(1;2)$. Right: A fundamental region for its convex core.}
  \label{Fig:MaxPinched}
\end{figure}
Note that face identifications glue boundaries of those curvilinear triangles in such a way that the convex core has boundary consisting of two totally geodesic 3-punctured spheres. Note also that this example has particularly simple geometry. There are (countably infinitely many) more complicated structures, with varied cusp shapes. More discussion can be found in~\cite[Section~4.3]{Dang15}; see also~\cite{KeenSeries}.

\section{Pinched handlebodies}

We will glue the compression bodies of the previous section to carefully constructed handlebodies, obtaining manifolds with a Heegaard splitting of genus $n+1$. To ensure there is not a Heegaard splitting of lower genus, we will use techniques of~\cite{BurtonPurcell} to bound the Hempel distance of the splitting. In this section we review the necessary background, particularly Hempel distance, and state the relevant results on a handlebody.

Let $S$ be a closed, oriented, connected surface, and let $\calC(S)$ denote the curve complex of $S$. Thus $\calC(S)$ is the simplicial complex with vertices corresponding to isotopy classes of simple closed curves, with a collection of $k+1$ curves forming a $k$ simplex when the curves can be realized disjointly on $S$. Distance in $\calC(S)$ is the minimum number of edges in any path in the 1-skeleton.

\begin{definition}\label{Def:HepelDist}
Suppose a compression body or handlebody has outer boundary $S$. The \emph{disk set} in $\calC(S)$ consists of vertices that correspond to boundaries of essential disks in the compression body. Thus a Heegaard splitting of a 3-manifold along $S$ has two disk sets, one on either side. The \emph{Hempel distance} of the Heegaard splitting is the minimal distance in $\calC(S)$ between the disk sets; see Hempel~\cite{Hempel}. We denote the disk set of a compression body or handlebody $C$ by $D(C)$.
\end{definition}

The Hempel distance is relevant to us because it can be used to show a manifold has tunnel number $n$: Scharlemann and Tomova showed that if the Hempel distance of a genus $g$ Heegaard splitting is strictly greater than $2g$, then the manifold will have a unique Heegaard splitting of genus $g$ and no Heegaard splitting of smaller genus~\cite{ScharlemannTomova}. We wish to build manifolds with Heegaard splittings of high distance. 

Suppose $P$ is a disjoint collection of curves on $S$. Recall that a \emph{multitwist} on $P$ is a collection of Dehn twists, one for each curve of $P$. 

\begin{lemma}\label{Lem:Handlebody}
Let $C_0$ denote the maximally pinched structure on $C(1;n+1)$ of \reflem{maxCuspOnC}. Let $P$ denote the pinched curves on the outer boundary $S$; note these form a maximal simplex of $\calC(S)$. 

\begin{enumerate}
\item There exists a maximally pinched structure $H_0$ on a handlebody $H$ of genus $n+1$ such that the pinching is along curves of $P\subset S\cong \bdy H$.

\item More generally, assuming results of~\cite{MaherSchleimer}, there exists a maximally pinched structure $H_0$ on a handlebody $H$ of genus $n+1$ with pinched curves $P\subset S$ such that if $T$ is a multitwist on $P$, and $D_0$ is the disk set of $H_0$, then the distance between disk sets of $C_0$ and of $T(D_0)$ is at least $2n+3$.
\end{enumerate}
\end{lemma}

\begin{proof}
The first part of the lemma is straightforward. For example, take any homeomorphism of $C$ to the compression body in a Heegaard splitting of any one-cusped 3-manifold of genus $n+1$, and mark the curves $P$ on $S$. Let $H$ be the handlebody of the splitting, also with curves $P$ marked on $S$. Then the fact that there is a maximally pinched structure on $H$ with curves of $P$ pinched to parabolics follows from Thurston's Uniformization theorem (see Morgan~\cite{Morgan}). 

The second part follows as in the proof of Theorem~4.12 of~\cite{BurtonPurcell}. Let $\calH(S)$ denote the handlebody graph, consisting of a vertex for each handlebody with boundary $S$ and an edge between handlebodies whose disk sets intersect in the curve complex $\calC(S)$. The distance between handlebodies in $\calH(S)$ is defined to be the minimal number of edges in a path between them in $\calH(S)$. Abusing notation slightly, we let $D\from \calH(S)\to\calC(S)$ be the relation taking a handlebody $V \in \calH(S)$ to the disk set $D(V)$ of $V$. Let $h$ be the relation taking a disk set $D(C)$ of a compression body $C$ with outer boundary $S$ to the handlebody complex $\calH(S)$, such that $h$ maps $D(C)$ to the set $h(D(C))\subset \calH(S)$ consisting of all $V\in\calH(S)$ for which $D(C)$ is a subset of $D(V)$. In particular, $D(h(D(C_0)))$ is a subset of $\calC(S)$. Let $X_0$ be the set
\[ X_0 = \{ h(T(D(C_0))) \in \calH(S) \mid T\mbox{ is a multitwist on } P\}, \]
so $X_0$ is the subset of $\calH(S)$ consisting of all handlebodies whose disk sets contain $T(D(C_0))$ for some multitwist $T$.

Now, a Dehn twist along a curve in $P$ is an isometry of $\calC(S)$ fixing $P$ pointwise. Thus if $T$ is any multitwist on $P$, the distance from $D(C_0)$ to $P$ is equal to the distance from $T(D(C_0))$ to $P$. Call this distance $R_0$. 

Then $X_0$ has bounded diameter in $\calH(S)$, for if $V,W\in X_0$, then there are multitwists $T_1$ and $T_2$ such that $T_1(D(C_0))\subset D(V)$ and $T_2(D(C_0))\subset D(W)$. Then
\begin{align*}
  d(D(V),D(W)) \leq & d(D(V),P)+d(P, D(W)) \\
  \leq & d(T_1(D(C_0)),P) + d(P,T_2(D(C_0))) =2R_0, 
\end{align*}
where $d(\cdot,\cdot)$ denotes minimal distance in $\calC(S)$. 
The distance in the handlebody graph is bounded by one more than the distance in $\calC(S)$; see~\cite[Lemma~4.11]{BurtonPurcell}. Hence $X_0$ has bounded diameter in $\calH(S)$.

By work of Maher and Schleimer, $\calH(S)$ has infinite diameter~\cite{MaherSchleimer}. Thus we may choose a handlebody $Y$ in $\calH(S)$ such that for all $V\in X_0$, the distance in the handlebody graph between $Y$ and $V$ is at least $2n+3$.

Now, take a maximally pinched structure on $Y$ with curves $P$ on $\bdy Y$ pinched; again such a structure exists by Thurston's Uniformization Theorem~\cite{Morgan}. 
\end{proof}

Maximally pinched structures on handlebodies have been studied in depth. For example, Keen and Series examined maximally pinched structures on genus~2 handlebodies, and showed they are obtained by solutions to polynomials, one for each curve on the 4-punctured sphere, each giving unique geometry~\cite{KeenSeries}.

More generally, the limit sets of maximally pinched structures are known to be circle packings, given by limit sets of 3-punctured spheres that form the boundary at infinity of such a manifold~\cite{KeenMaskitSeries}. When we take the convex core of such a structure, we cut out the half-space bounded by each circle in the circle packing, obtaining in the quotient a finite volume space whose boundary consists of embedded, totally geodesic 3-punctured spheres. 

The interiors of these convex cores can have a wide range of geometry. As one example, the universal cover of a maximally pinched structure on a genus~$2$ handlebody is shown in~\cite[page~200]{MSW:IndrasPearls}. There are four red circles shown the picture, and a fundamental domain for the manifold is obtained by cutting out hyperplanes bounded by these circles; the manifold is formed by gluing the boundaries of those hyperplanes in pairs. See also Figure~7.7 in~\cite[page~205]{MSW:IndrasPearls}. The convex core of the manifold is then obtained by cutting out hyperplanes bounded by ``dual circles,'' enclosing the black curvilinear triangles, and an additional circle enclosing all three black triangles. Note that a fundamental domain for the convex core is therefore a regular ideal octahedron.

\section{Cusp shapes of tunnel number $n$ manifolds}

In this section, we prove the main theorems.

\begin{proof}[Proof of \refthm{TunnelNCuspShapes}]
Let $C_0$ be the compression body with the maximally pinched structure of \reflem{maxCuspOnC}. Let $H_0$ be a maximally pinched structure on a genus $n+1$ handlebody from \reflem{Handlebody}. The convex cores of $C_0$ and $H_0$ have boundary consisting of embedded, totally geodesic 3-punctured spheres. There is a unique hyperbolic structure on such a 3-punctured sphere (see~\cite{Adams-85}), hence the convex cores of $C_0$ and $H_0$ can be glued along their 3-punctured sphere boundaries via isometry.

Since both $C_0$ and $H_0$ are geometrically finite, their convex cores have finite volume and the resulting manifold $\widehat{M}$ is a complete, finite volume hyperbolic manifold. The manifold $\widehat{M}$ has a system of $n$ tunnels, as it is the result of gluing a $(1;n+1)$-compression body and a genus $n+1$ handlebody along their genus $n+1$ boundaries; let $S$ denote this surface of genus $n+1$. Furthermore, $\widehat{M}$ has $(3n+4)$ rank-2 cusps, $(3n+3)$ of which come from gluing the rank-1 cusps on the boundaries of the convex cores of $C_0$ and $H_0$. Denote these cusps by $T_1$, $T_2$, $\dots$, $T_{3n+3}$. The remaining cusp is the rank-2 cusp of $C_0$ whose shape is $[T_{\epsilon/2}]$, which is $\epsilon/2$-close to $[T]$ (from \reflem{maxCuspOnC}).

For each $i=1, \dots 3n+3$, the cusp $T_i$ meets the Heegard surface $S$ of genus $n+1$. Choose a homology basis $(\mu_i,\lambda_i)$ for $\bdy T_i$ such that $\lambda_i$ is parallel to the slope of $S$ on $T_i$, and $\mu_i$ is any curve with intersection number $1$ with $\lambda_i$. Let $s_i$ denote the slope $s_i = \mu_i + k_i\lambda_i$, where $k_i$ is a positive integer. Let $k=(k_1, \dots, k_{3n+3})$. Let $M_k$ denote the manifold obtained by Dehn filling $\widehat{M}$ along the slopes $s_i$ on cusps $T_i$. Then $M_k$ has Heegaard surface $S$. On one side of $S$ is a genus $(n+1)$ handlebody, and on the other side is a $(1;n+1)$-compression body. Thus $M_k$ is manifold with a system of $n$ tunnels. Moreover, $M_k$ has only one rank-$2$ cusp by construction.

Further, note that Dehn filling along slopes $s_i$ fixes the disk set of $H_0$, but modifies the disk set of $C_0$ by applying a multitwist along $P$. Assuming the second part of \reflem{Handlebody}, the Hempel distance of the resulting Heegaard splitting of $M_k$ will be larger than $2n+2$. Hence by work of Scharlemann and Tomova~\cite{ScharlemannTomova}, this is the minimal genus Heegaard splitting of $M_k$. (If we do not assume the second part of \reflem{Handlebody}, we only conclude that $M_k$ has a system of $n$ tunnels.)

Let $k_i\to\infty$ for each $i$. Then $k\to \infty$ and by the hyperbolic Dehn surgery theorem, the manifolds $M_k$ form a sequence of complete, finite volume, hyperbolic manifolds that converges geometrically to $\widehat{M}$. It follows that the rank $2$ cusps of $M_k$ converge to the rank $2$ cusp of $\widehat{M}$, which has shape $[T_{\epsilon/2}]$. Therefore, by choosing large enough $k$, the manifold $M=M_k$ has a system of $n$ tunnels and cusp shape $\epsilon$-close to $[T]$. Assuming the second part of \reflem{Handlebody}, this manifold is tunnel number $n$. 
\end{proof}

We may actually complete the proof of \refthm{Tunnel1CuspShapes} without using the work of Maher and Schleimer~\cite{MaherSchleimer} used in \reflem{Handlebody}.

\begin{proof}[Proof of \refthm{Tunnel1CuspShapes}]
Let $n=2$. For any $\epsilon>0$, and any $[T]$, by the proof of \refthm{TunnelNCuspShapes} we obtain a hyperbolic 3-manifold $M$ with a single tunnel and cusp shape $\epsilon$-close to $[T]$. Because there is a single tunnel, $M$ has a genus $2$ Heegaard splitting. Since there are no hyperbolic manifolds with smaller genus Heegaard splittings, $M$ must have tunnel number one. 
\end{proof}

\begin{corollary}\label{Cor:TunnelNExceptions}
For any integer $n\geq 1$, there exist infinitely many hyperbolic tunnel number~$n$ manifolds with at most one exceptional Dehn filling.
\end{corollary}

\begin{proof}
For any $R>7$ and $0<\epsilon<1/2$, let $[T]$ be a flat torus that is rectangular in shape, normalized to have one side of length $1$ and the perpendicular side of length $R$. Let $M_R$ be a tunnel number one manifold whose cusp shape is within $\epsilon$ of $[T]$, provided by \refthm{TunnelNCuspShapes}. In a maximal cusp neighborhood, the shortest curve on the cusp has length at least $1$. Thus the cusp embeds as a Euclidean torus with one slope length at least $1$, and all other slope lengths at least $R-\epsilon > 7-\epsilon>6$. By the 6-Theorem~\cite{Agol:Bounds, Lackenby:Word}, there can be at most one exceptional Dehn filling of $M_R$.
\end{proof}

\section{Exceptional fillings}\label{Sec:Exceptional}

Recall that Berge knots have tunnel number one. \refcor{TunnelNExceptions}
contrasts with the following proposition concerning Berge knots, which follows from work of Baker~\cite{Baker:LargeVolBerge}.

\begin{proposition}\label{Prop:Berge}
  Let $[T_0]$ denote the Euclidean similarity structure on the torus that includes the $1\times 2$ rectangle. Let $U$ denote a neighborhood of $[T_0]$. Then there exists a sequence of hyperbolic Berge knots, all with cusp shape lying in the neighborhood $U$ of $[T_0]$.
  Moreover, maximal cusps for these Berge knots limit to a $2\times 4$ rectangle. 
\end{proposition}

\begin{proof}
Baker showed that large volume Berge knots are obtained by Dehn filling the complement of a minimally twisted $(2n+1)$-chain link $L_{2n+1}$~\cite{Baker:LargeVolBerge}. For $V$ large, any Berge knot with volume at least $V$ is obtained by such a Dehn filling. Thus for fixed $n>0$, there is a sequence of such knots converging geometrically to $S^3\setminus L_{2n+1}$, and so the cusp shapes of this sequence of knots converge to a cusp shape of $S^3\setminus L_{2n+1}$.

We show that as $n\to\infty$, cusp shapes of $S^3\setminus L_{2n+1}$ approach $[T_0]$. By symmetry, cusp shapes of all cusps of $S^3\setminus L_{2n+1}$ are identical, so we need only show that one cusp shape approaches this similarity structure. 

The minimally twisted $(2n+1)$-chain link is obtained by Dehn filling a link obtained by the following procedure. First, augment the minimally twisted 3-chain link; the complement is a hyperbolic manifold with four cusps, three of which have shape $[T_0]$. Denote the link by $L_3^a$. See \reffig{Augmented3Chain}, obtained from SnapPy~\cite{SnapPy}. 
\begin{figure}
  \includegraphics[height=1.2in]{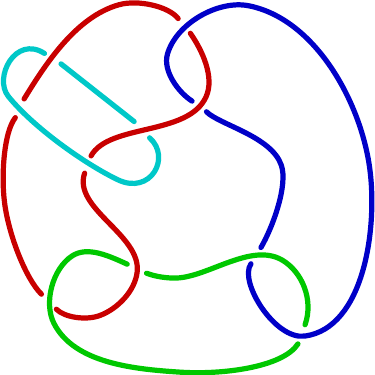}
  \hspace{.5in}
  \includegraphics[height=1.2in]{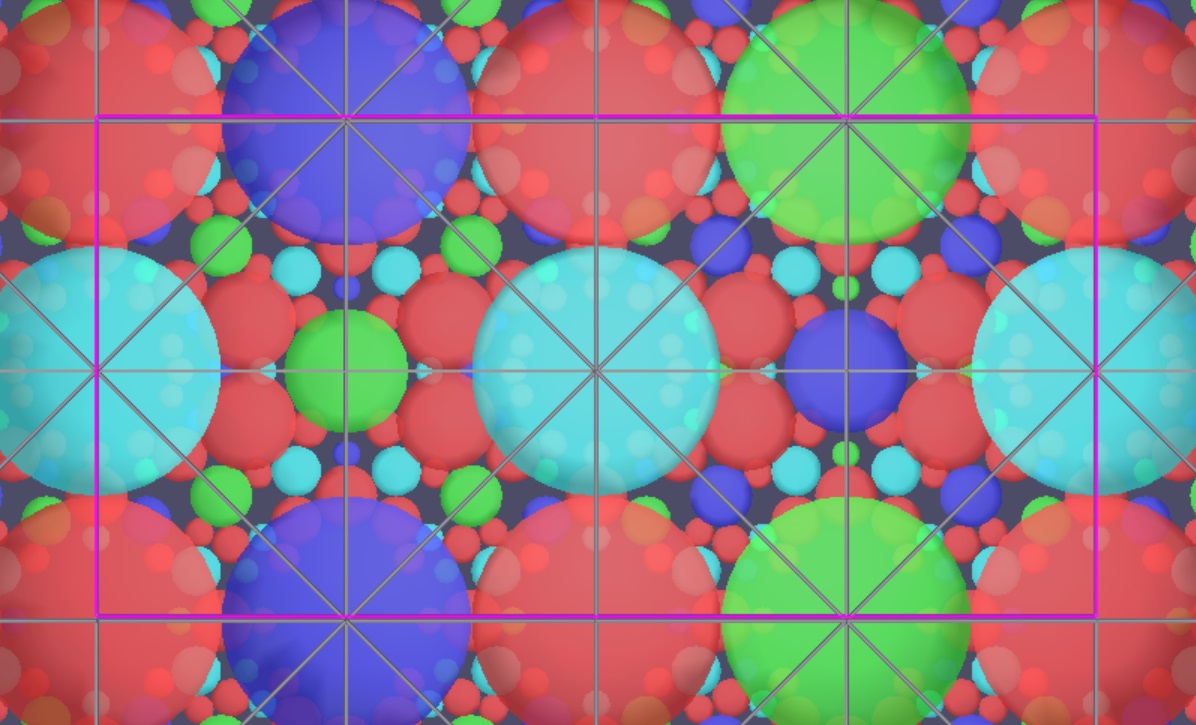}
  \caption{The augmented minimally twisted 3-chain link $L_3^a$ is shown on the left. It is hyperbolic, with each chain link component having the cusp shape of a $1\times 2$ rectangle, up to scale. One such cusp is shown on the right.}
  \label{Fig:Augmented3Chain}
\end{figure}
Next, take the augmented minimally twisted 2-chain link, denoted $L_2^a$, shown in \reffig{Augmented2Chain}. This is a hyperbolic manifold with three cusps, two of which have shape $[T_0]$.
\begin{figure}
  \includegraphics[height=1.2in]{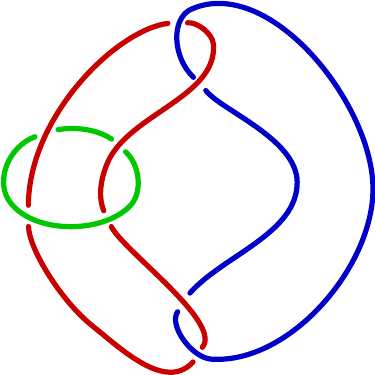}
  \hspace{.5in}
  \includegraphics[height=1.2in]{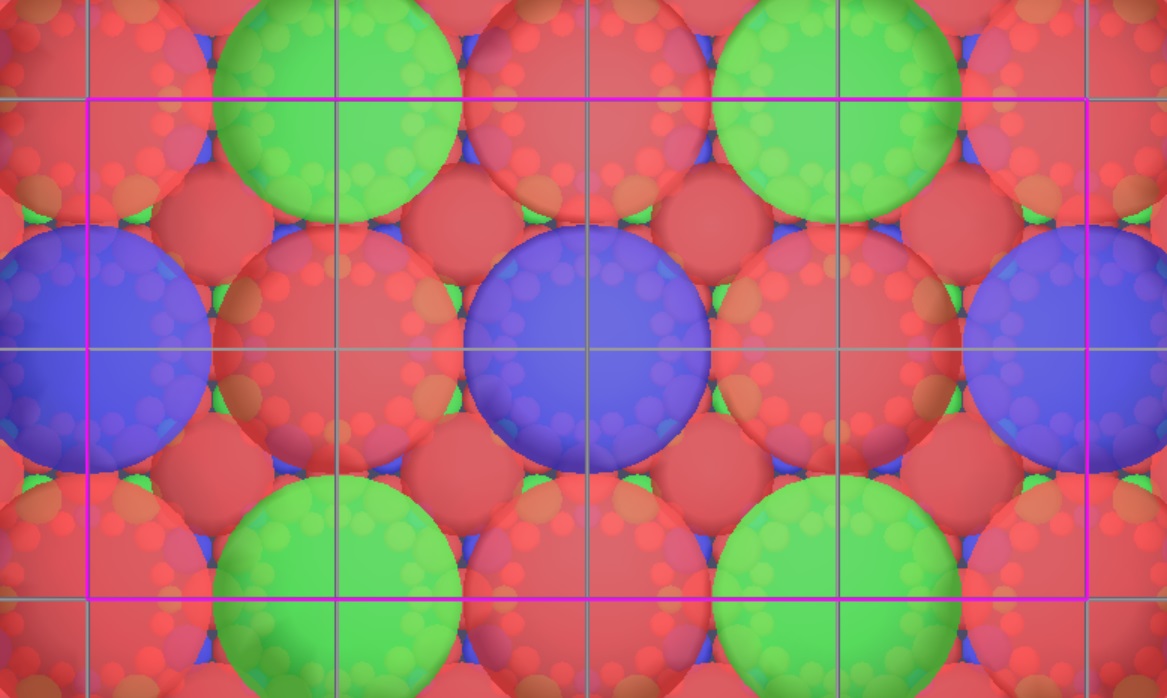}
  \caption{The augmented minimally twisted 2-chain link $L_2^a$ on the left. It is hyperbolic. Its two cusps coming from the chain, not the augmentation component, have cusp shape $[T_0]$.}
  \label{Fig:Augmented2Chain}
\end{figure}

We will take belted sums, as in Adams~\cite{Adams-85}. Recall that belted sum cuts two manifolds along embedded 3-punctured spheres and reglues by isometry of the 3-punctured spheres. Take the belted sum of $S^3\setminus L_3^a$ and $S^3\setminus L_2^a$, summing along the augmentation component of each. The result will be the augmented minimally twisted 5-chain link $L_5^a$. 
For both $L_3^a$ and $L_2^a$, the 3-punctured sphere we cut along meets the cusp under consideration exactly twice in a curve perpendicular to the cusp, with the two curves spaced equidistantly along the cusp. In \reffig{Augmented3Chain}, the 3-punctured sphere corresponds to the two curves running vertically across the cusp over the horoball colored cyan. Note they are equidistantly spaced and perpendicular to the boundary. In \reffig{Augmented2Chain}, the 3-punctured sphere runs over the green horoball. When we perform the belted sum, we cut along these curves and reassemble, gluing half of a cusp of $L_3^a$ to half of a cusp of $L_2^a$ to obtain a cusp of $L_5^a$, shown in \reffig{Augmented5Chain}. Because the cut happened halfway along each, along 3-punctured spheres meeting full-sized horospheres, the cusp shape is still $[T_0]$. 

\begin{figure}
  \includegraphics[height=1.5in]{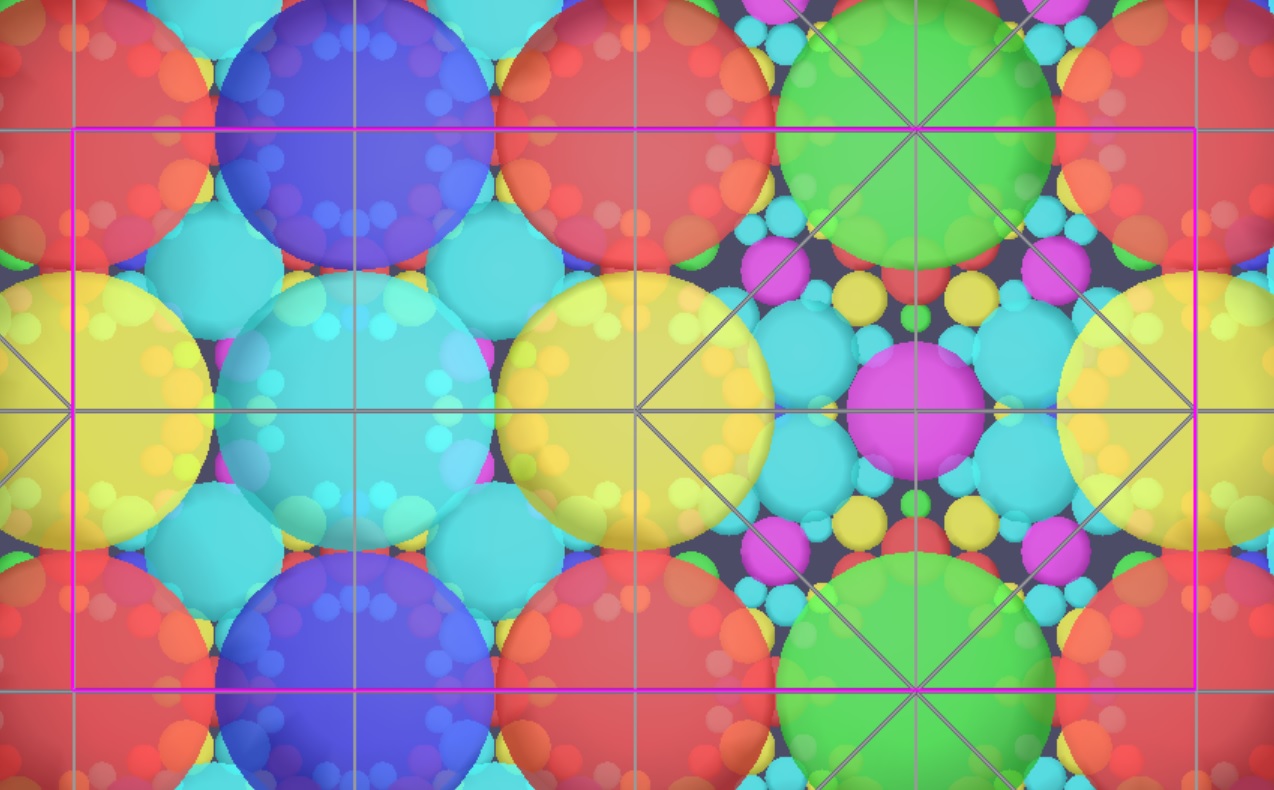}
  \caption{The cusp of a chain component of the minimally twisted 5-chain link is obtained by taking half of that of $L_3^a$ and half of that of $L_2^a$.}
  \label{Fig:Augmented5Chain}
\end{figure}

The cusp of the augmentation component under the belted sum is similarly changed, now by stacking the cusps corresponding to the augmentation components of $L_3^a$ and $L_2^a$ along 3-punctured spheres. We do not need to know the geometry of this cusp precisely, except that the length of the meridian of the cusp has increased. 

Now to obtain the augmented minimally twisted $(2n+1)$-chain link $L_{2n+1}^a$, inductively take the belt sum of $L_{2n-1}^a$ with $L_2^a$. As in the case of $L_5^a$, the cusp shape of the chain link component is obtained by cutting cusps of $L_{2n-1}^a$ and $L_2^a$ in half and assembling half of each into the new cusp. Hence its shape is still $[T_0]$. The cusp of the augmentation component is obtained by stacking the cusp corresponding to the augmentation components of $L_{2n-1}^a$ and $L_2^a$; this is equivalent to stacking one copy of the cusp of $L_3^a$ and $n-1$ copies of the cusp of $L_2^a$. Thus the meridian length grows linearly with $n$.

Note that Dehn filling the meridian of the augmentation component of $L_{2n+1}^a$ gives the minimally twisted chain link $L_{2n+1}$. Since the length of the meridian grows linearly with $n$, by the hyperbolic Dehn filling theorem for $n$ large enough, the cusp shape of $L_{2n+1}$ will lie in the neighborhood $U$ of $[T_0]$. Take a sequence of filling slopes on the other cusps of $L_{2n+1}$ that have high length and yield a Berge knot. Again the hyperbolic Dehn filling theorem implies that the cusp shape of the Berge knot lies in $U$.

Finally, because Berge knots are obtained by Dehn filling, which gives a geometrically convergent sequence, not only will their cusp shapes converge, but in addition, their maximal cusps will converge to the maximal cusp of the unfilled component of the augmented minimally twisted $(2n+1)$-chain link. The maximal cusp of a chain link component of $L_3^a$ and $L_2^a$ is a $2\times 4$ rectangle. (This can be shown via SnapPy~\cite{SnapPy} as in Figures~\ref{Fig:Augmented3Chain} and~\ref{Fig:Augmented2Chain}, or by hand.) The cutting and regluing in both cusps is along 3-punctured spheres meeting that maximal cusp in identical horocycle patterns. Thus the belted sums of these links have cusps that are $2\times 4$ rectangles, and the result follows. 
\end{proof}

\begin{remark}
As noted by the referee, in the proof of \refprop{Berge}, the hyperbolic structure on the augmented, minimally twisted $(2n-1)$-chain link complement can also be obtained by taking the $(2n-1)$-fold cover of the Whitehead link, then cutting along the 3-punctured sphere bounded by the augmented component, untwisting, and then regluing. Following this procedure using the geometry of the Whitehead link complement, for example as in~\cite{KPR}, leads to the same results on cusp shapes without appealing to belted sums.
\end{remark}

\bibliography{references}
\bibliographystyle{amsplain}

\end{document}